\documentclass{amsart}
\usepackage{amsmath}
  \usepackage{paralist}
\usepackage{amsfonts}
\usepackage{amsthm}
\usepackage{amssymb}
\usepackage{txfonts}
\usepackage{amscd}
\usepackage{graphicx,subfigure}
 \usepackage[colorlinks=true]{hyperref}
\hypersetup{urlcolor=blue, citecolor=red}

\newtheorem{theorem}{Theorem}[section]

\newtheorem{lemma}[theorem]{Lemma}
\newtheorem{proposition}{Proposition}

\theoremstyle{definition}
\newtheorem{definition}[theorem]{Definition}
\newtheorem{remark}{Remark}

\newcommand{\ep}{\varepsilon}

\title[Computing the Minimal Average Energy]
      {Perturbation and Numerical Methods
for Computing the Minimal Average Energy}

\author[Timothy Blass and Rafael de la Llave]{}


\email{tblass@math.utexas.edu}
\email{llave@math.utexas.edu}

\thanks{This work has been
supported by NSF grant DMS 0901389 and Texas Coordinating
board NHARP 0223}

\begin{document}
\maketitle

\centerline{\scshape Timothy Blass and Rafael de la Llave}
\medskip
{\footnotesize
 \centerline{Department of Mathematics, 1 University Station C1200}
   \centerline{Austin, TX 78712-0257, USA}
} 

%

\bigskip


\begin{abstract}
We investigate the differentiability of minimal average
energy associated to the functionals
$S_\ep (u) = \int_{\mathbb{R}^d} \frac{1}{2}|\nabla u|^2 + \ep V(x,u)\, dx$,
using numerical and perturbative methods. We use
the Sobolev gradient descent method as a numerical tool to
compute solutions of the Euler-Lagrange equations
with some periodicity conditions; this is
the cell problem in homogenization.
We use these solutions to determine the average minimal energy
as a function of the slope.
We also obtain a representation of the solutions to the Euler-Lagrange
equations as a Lindstedt series in the perturbation parameter
$\ep$, and use this to confirm our numerical results. Additionally, we
prove convergence of the Lindstedt series.
\end{abstract}


\section{Introduction}
Let $d \in \mathbb{N}$ be fixed, and let $V: \mathbb{R}^d \times \mathbb{R} \to \mathbb{R}$, be periodic
under integer translations. That is $V(x+k,y+l) = V(x,y)$ for all
$(k,l)\in \mathbb{Z}^d \times \mathbb{Z}$, where $(x,y)=(x_1,\ldots , x_d,y)\in \mathbb{R}^d \times \mathbb{R}$.
Furthermore, assume $V$ is analytic.
We consider the formal variational problem
\begin{equation} \label{principle}
S_{\! \ep} (u(x)) = \int_{\mathbb{R}^d} \frac{1}{2}|\nabla u(x)|^2 + \ep V(x,u(x))\, dx,
\end{equation}
where $\ep$ is a small parameter, so that $S_{\! \ep}$ is a small
perturbation of $S_{\! 0}(u) = \int_{\mathbb{R}^d} \frac{1}{2}|\nabla u|^2\, dx$.
We call $u \in H^1_{loc}(\mathbb{R}^d)$ a minimizer (or minimal solution)
of $S_{\! \ep}$ if for all $\phi \in H^1_{comp}(\mathbb{R}^d)$
\[
\int_{supp(\phi)} \frac{1}{2}|\nabla (u+\phi )|^2 + \ep V(x,u+\phi )
-\frac{1}{2}|\nabla u|^2 - \ep V(x,u) \, dx \geq 0.
\]
Any minimizer of (\ref{principle}) must solve the Euler-Lagrange equation
\begin{equation} \label{EL}
-\Delta u + \ep V_y(x,u) =0,
\end{equation}
and by standard elliptic regularity theory will be at least as regular as $V$,
so analytic in our case.
\begin{definition}
Following \cite{Moser86}, we say a continuous function $u$ is
\emph{non-selfintersecting} if the  graph of $u$ does
not intersect integer translates of itself. That is, if
$u \in C^0(\mathbb{R}^d,\mathbb{R})$ and
 $\forall (k,l) \in \mathbb{Z}^d \times \mathbb{Z}$
\begin{equation} \label{birkhoff}
u(x+k) + l -u(x) >0, \,\,\mbox{or}\,\, <0, \,\, \mbox{or} \equiv 0,
\end{equation}
where the three alternatives are independent of $x$.
\end{definition}
This is also referred to in the literature
as the \emph{Birkhoff} property.
From a geometric viewpoint, this means  that the graph of $u$ projects
into $\mathbb{T}^{d+1} = \mathbb{R}^{d+1}/\mathbb{Z}^{d+1}$
without intersecting itself, unless it coincides exactly.

If $u$ is non-selfintersecting, then there is a
\emph{rotation vector}, $\omega \in \mathbb{R}^d$, associated to $u$ such that
\begin{equation} \label{planelike}
\sup_{x\in \mathbb{R}^d}|u(x) - \omega \cdot x| < \infty
\end{equation}
\cite{Moser86}. A function $u$ satisfying (\ref{planelike}) is
called \emph{plane-like} because its graph is at a bounded
distance from the a hyperplane
in $\mathbb{R}^{d+1}$ with normal vector $(\omega, -1)$.

We denote the set of all non-selfintersecting minimizers of $S_{\! \ep}$
with rotation vector $\omega$ by $\mathcal{M}_{\omega}(S_{\! \ep})$, and more briefly
as $\mathcal{M}_{\omega}$, where dependence on $S_{\! \ep}$ is understood. As shown
in \cite{Moser86}, $\mathcal{M}_{\omega}$ is nonempty for all $\omega \in \mathbb{R}^d$.

The rational dependency of $\omega$ can play a role in the structure
of $\mathcal{M}_{\omega}$. The case where $\bar{\omega} = (\omega,-1)$ is rationally independent
was studied in \cite{Bangert87a}, and the
rationally dependent case in \cite{Bangert89}.

We define the minimal average energy $A_{\varepsilon} : \mathbb{R}^d \to \mathbb{R}$ by
\begin{equation} \label{Ae}
A_{\varepsilon} (\omega) = \lim_{R \to \infty} \frac{1}{|B_R|}\int_{B_R}\frac{1}{2}|\nabla u|^2
+ \ep V(x,u)\, dx,
\end{equation}
where $u\in \mathcal{M}_{\omega}$, $B_R=\{x\in \mathbb{R}^d : |x| \leq R\}$.
It was shown in \cite{Senn91} that $A_{\varepsilon}$ is well-defined,
that is, the limit exists and is independent of the choice of
minimizer $u \in \mathcal{M}_{\omega}$.
Furthermore, $A_{\varepsilon}$ is convex, so that one-sided derivatives of $A_{\varepsilon}$
exist at each $\omega\in \mathbb{R}^d$, \cite{Senn91}. In fact, for $\ep =0$,
any $u\in \mathcal{M}_{\omega}$
has the form $u(x) = \omega \cdot x + \alpha$ for some $\alpha \in \mathbb{R}$ (see
\cite{Moser86}).
Thus, $A_0(\omega)=\frac{1}{2}|\omega |^2$ is smooth. However, for typical $V$
the differentiability of $A_{\varepsilon}$ breaks down when
$\ep >0$, and for large
enough $\ep$ the set of points where $A_{\varepsilon}$ is not differentiable will
be dense in $\mathbb{R}^d$ (see \cite{Bangert87b}).
Because $A_{\varepsilon}$ has one-sided derivatives
the graph of $A_{\varepsilon}$ will have ``corners''
(i.e. the jump in the gradient of $A_{\varepsilon}$ in a direction $e_j$)
at points of nondifferentiability. In \cite{Senn95}, page 356 there is a
formula for the one-sided directional derivative of $A_{\varepsilon}$ involving
special types of minimizers of $S_{\! \ep}$, described in Section \ref{Pert}.

In this paper, we present a numerical approach to computing
solutions of (\ref{EL}) via a gradient descent method known
as the Sobolev gradient, as explained in Section \ref{Numerics}.
We use this to compute $A_{\varepsilon}$ and
$D_{e_j}A_{\varepsilon}(\omega) + D_{-e_j}A_{\varepsilon}(\omega)$, that is
the jump in the gradient of $A_{\varepsilon}$ in the direction $e_j$. We will
sometimes refer to a jump in the gradient as a ``corner''.
In Section \ref{Pert} we present a perturbation method
for finding the minimizers
of $S_{\! \ep}$ needed to apply the formula provided in \cite{Senn95}
to compute $D_{e_j}A_{\varepsilon}(\omega) + D_{-e_j}A_{\varepsilon}(\omega)$.

Throughout the paper, we will deal with the case $\omega\in \frac{1}{N}\mathbb{Z}^d$.
There are two good reasons for this. The first is that minimizers
of (\ref{principle}) with irrational rotation vectors
can be obtained as limits
of sequences of minimizers with rational rotation vectors.  The second reason
is that for rational $\omega$, the partial differential equations we need
to solve are well-defined on the torus $\mathbb{T}^d$. That is, they have
periodic boundary conditions. This allows the use of Fourier transforms
in the numerical method. It also makes each equation in (\ref{Lind})
of the form $\Delta \phi = g$, which can be solved for $\phi$
provided $g$ has average zero.

In this setting, for a fixed $\omega\in \frac{1}{N}\mathbb{Z}^d$,
 the formal functional in (\ref{principle}) can
be replaced by the reduced functional
\begin{equation} \label{S_red}
S_{\! \ep, N} (u) = \int_{N\mathbb{T}^d}\frac{1}{2}|\nabla u|^2 + \ep V(x,u)\, dx,
\quad u = \omega \cdot x + z(x)
\end{equation}
where $z$ is $N\mathbb{Z}^d$-periodic. Considering the limits of solutions
$u_N(x) = \omega_N x + z_N(x)$, $\omega_N\in \frac{1}{N}\mathbb{Z}^d$ as $N \to \infty$,
we see that this problem is related to periodic homogenization. Rescaling
so that (\ref{S_red}) is defined on the unit cube, the process of letting
$N\to \infty$ is the same as letting the $x$ dependence oscillate very
rapidly.


\section{Numerics} \label{Numerics}

In this section we investigate numerically the size of
the corners of $A_{\varepsilon}$ as $\ep$ varies. Since our problem comes
from a variational principle, steepest descent methods are a
natural approach. We consider a fixed $\omega\in \frac{1}{N}\mathbb{Z}^d$ and
seek solutions of (\ref{EL}) of the form $u(x) = \omega \cdot x + z(x)$,
where $z$ is $N$-periodic. To find such a function $u$, we solve
\begin{equation} \label{ELz}
\Delta z = \ep V_y(x,\omega \cdot x + z), \quad
z(x + Nk) = z(x), \, \forall k \in \mathbb{Z}^d,
\end{equation}
for $z$, and set $u(x) = \omega \cdot x + z(x)$.
In this setting, we see that $z$ is a critical point of
the \emph{reduced variational problem}
\[
S_{\! \ep, N} (z) = \int_{N\mathbb{T}^d}\frac{1}{2}|\nabla z|^2 + \ep V(x,\omega \cdot x + z)\, dx,
\]
where $N\mathbb{T}^d = \mathbb{R}^d/N\mathbb{Z}^d$.
The Fr\'echet derivative of $S_{\! \ep, N}$ at $z$, applied to
$\eta \in H^1$ is
\[
DS_{\! \ep, N} (z)\eta = \int_{\mathbb{T}^d} \nabla z\cdot \nabla \eta + \ep V_y(x,\omega \cdot x + z)
\eta \,dx.
\]
We recall that the gradient of $S_{\! \ep, N}$ with respect to a Hilbert space, $H$, is
the unique element of $h \in H$ such that
$DS_{\! \ep, N} (z)\eta = \langle g, \eta \rangle_H$ for all $\eta \in H$.
For instance, the $L^2$-gradient of $S_{\! \ep, N}$ at $z$ is the unique element of
$ L^2 = H^0$, which we will write as $\nabla_0 S_{\! \ep, N} (z)$, such that
$DS_{\! \ep, N} (u)\eta = \langle \nabla_0 S_{\! \ep, N} (z), \eta \rangle_{L^2}$.
Integrating by parts, we have,
$DS_{\! \ep, N} (z)\eta = \langle -\Delta z + \ep V_y(x,\omega \cdot x + z), \eta \rangle_{L^2}$.

Considering gradients with respect to other inner products
has been a fruitful endeavor (see \cite{Neuberger10}).
In our particular case, by considering
the gradient of $S_{\! \ep, N}$ in  $H^{\beta}$ for
$\beta \in (0,1]$ we avoid the stiffness of the problem that
appears in the $H^0$ case. This is explained in more detail in the
remark following the derivation of $\nabla_{\beta} S_{\! \ep, N} (z)$.

The standard inner product on $H^{\beta}(N\mathbb{T}^d)$ is
$\langle u, v \rangle_{H^{\beta}} = \langle (I - \Delta)^{\beta} u,v
\rangle_{H^0}$. For $\gamma > 0$ we define the inner product
$\langle u, v \rangle_{\beta} = \langle (\gamma I - \Delta)^{\beta} u,v
\rangle_{H^0}$, which determines a norm on $H^{\beta}$ that is equivalent
to the standard norm. For more details on this see the introduction of
\cite{BlassLlaveVal10}.
In that paper, the authors show that the descent equation
$\partial_t u = - \nabla_{\beta} S_{\! \ep, N} (u)$ satisfies a comparison
principle, and preserves the class of non-selfintersecting functions.
In this way, they obtain critical points of $S_{\! \ep, N}$ with rotation vector
$\omega$ by choosing initial condition $u(x,0) = \omega \cdot x$, equivalently,
$z(x,0) = u(x,0) - \omega \cdot x = 0$.

We note that different choices of $\gamma$ result in different inner
 products, and
therefore different gradients of $S_{\! \ep, N}$. Thus, $\nabla_{\beta} S_{\! \ep, N} (u)$
depends
not only on the choice of Hilbert space, but also on the choice of
inner product
on that space.
We calculate the $H^{\beta}$-gradient as follows:
\begin{eqnarray*}
DS_{\! \ep, N}(z)\eta &=& \int_{N\mathbb{T}^d} \nabla z\cdot \nabla \eta +\ep V_y(x,\omega \cdot x +z)\eta
\,dx\\
&=& 
\left\langle -\Delta z +\ep V_y(x,\omega \cdot x + z), \eta \right\rangle_{L^2} \\
&=& \left\langle (\gamma I -\Delta)^{\beta}(\gamma I-\Delta)^{-\beta}
(-\Delta z +\ep V_y(x,\omega \cdot x + z)) ,\eta
\right\rangle_{L^2} \\
&=&  \left\langle (\gamma I -\Delta)^{-\beta}(\gamma z -\Delta z -\gamma z +
\ep V_y(x,\omega \cdot x + z)) ,\eta \right\rangle_{{\beta}} \\
&=&  \left\langle (\gamma I -\Delta)^{1-\beta}\, z -
(\gamma I -\Delta)^{-\beta}
(\gamma z -\ep V_y(x,\omega \cdot x + z)) , \eta \right\rangle_{{\beta}}.  \\
\end{eqnarray*}
Thus, our steepest descent equation in $H^{\beta}$, $\partial_t z =
-\nabla_{\beta} S_{\! \ep, N} (z)$, becomes
\begin{equation} \label{SobGrad}
\partial_t z = -\left(\gamma I -\Delta \right)^{1-\beta}z +
\left(\gamma I -\Delta \right)^{-\beta}
\left(\gamma z -\ep V_y(x,\omega \cdot x + z) \right),
\end{equation}
$z:[0,N]^d \to \mathbb{R}$ with periodic boundary conditions.

If the solution $z(x,t)$ of
(\ref{SobGrad}) approaches a critical point of $S_{\! \ep, N}$, that is $z(x,t) \to z_c(x)$ as
$t \to \infty$, then $z_c$ will solve $(\gamma+\Delta)^{1-\beta}\, z_c =
(\gamma+\Delta)^{-\beta}(\gamma z_c - V_y(x,\omega \cdot x +z_c))$, which reduces to
$-\Delta z_c + \ep V_y(x,\omega \cdot x + z_c) = 0$. Then $u_c = \omega \cdot x + z_c$ will
solve (\ref{EL}) and $\sup_x |u(x) - \omega \cdot x| < \infty$.
\begin{remark}
The stiffness of (\ref{SobGrad}) is significantly reduced for values of
$\beta \approx 1$. For instance,  the $L^2$ gradient descent equation
would be $\partial_t u = \Delta u - \ep V_y(x,u)$.
If we set $G(k)=\mathcal{F}\{-\ep V_y(x,u)\}(k)$, then
the descent equation in the Fourier domain becomes the system
$\partial_t \hat{u}_k = |k|^2\hat{u}_k + G(k)$ for each $k$.
As $|k|$ grows, these equations become increasingly stiff (the
eigenvalues of the linearization have a large spread). However,
the $H^{\beta}$ gradient descent equation in the Fourier domain becomes
$\partial_t \hat{u}_k = (\gamma + |k|^2)^{1-\beta}\hat{u}_k +
(\gamma + |k|^2)^{-\beta}G(k)$. So for $\beta \approx 1$ the stiffness
is greatly reduced.
\end{remark}

\subsection{Implementation of Sobolev gradient descent} \label{implement}
We fix $\omega\in \frac{1}{N}\mathbb{Z}^d$ and seek $z:[0,\infty)\times [0,N]^d \to \mathbb{R}$
solving (\ref{ELz}) with periodic boundary conditions. Thus, (\ref{ELz})
can be rewritten in the frequency domain via the Fourier transform.
The main benefit of this is that the pseudo-differential operators
in (\ref{ELz}) simplify greatly in the frequency domain. However,
the composition $V_y(x,\omega \cdot x + z(x))$ is complicated by the Fourier transform.
We take advantage of the simplicity of the operators in the frequency
domain in our numerical scheme, and pay the price of computing an
inverse fast Fourier transform at each time-step. This is explained
in the remark following equation (\ref{NumEq}).

We now develop the details of the implementation for $d=2$.
We know that $z(x,t) = u(x,t) - w\cdot x$ will be $N-$periodic if
we have $N\omega \in \mathbb{Z}^2$ and if $z(x,0) = u(x,0) - \omega \cdot x$
is an $N-$periodic function.
So we set $z_0(x) = z(x,0) = 0$, and
consider the equation satisfied by $\hat{z}$,
the Fourier transform of $z$.
We get
\begin{equation}
\hat{z}_t(t,\xi) = -(\gamma+|\xi|^2)^{1-\beta}\, \hat{z}(t,\xi) +
(\gamma+|\xi|^2)^{-\beta}(\gamma \, \hat{z}(t,\xi) -
\mathcal{F} [V_y(x,\omega \cdot x + z)](t,\xi)),
\end{equation}
where we have also used $\mathcal{F}[\cdot ]$ to denote the Fourier transform.
We now choose a time step $\Delta t$ and set $t_n=n \Delta t$.
We also break up the domain $[0,N]^2$ into $m^2$ discrete points
and represent $z(t_n,x)$ as an $m \times m$ array $z_n(i,j)$.
So, representing the discrete fast Fourier transform as $\mbox{fft}[\cdot]$
(and with $\hat{z}_n = \mbox{fft}[z_n]$) equation (3) is approximated as
\begin{equation}
\frac{\hat{z}_{n+1}-\hat{z}_n}{\Delta t} = -(\gamma+|\xi|^2)^{1-\beta}
\hat{z}_{n+1} + (\gamma+|\xi|^2)^{-\beta}(\gamma \hat{z}_{n+1} -
\mbox{fft}[V_y(x,\omega \cdot x  + z_n)]).
\end{equation}
This is a quasi-implicit method, since nearly all of the right-hand-side is
evaluated at the later time $t_{n+1}$.
Only the nonlinear term is evaluated at
time $t_n$, so we can easily solve for $\hat{z}_{n+1}$:
\begin{equation}\label{NumEq}
\hat{z}_{n+1}(i,j) = \frac{\hat{z}_n(i,j) - \Delta t (\gamma+\xi_1(i)^2 +
\xi_2(j)^2)^{-\beta}\mbox{fft}[V_y(x,\omega \cdot
x(i,j)+z_n(i,j))]}{1+\Delta t
(\gamma  + \xi_1(i)^2 + \xi_2(j)^2)^{1-\beta}-\gamma \Delta t
(\gamma + \xi_1(i)^2+\xi_2(j)^2)^{-\beta}}.
\end{equation}
\begin{remark}
After the $n$-th step, we have $\hat{z}_n$. In order to compute
$V_y(x,\omega \cdot x+z_n)$ we apply the inverse fast
Fourier transform to get $z_n= \mbox{ifft}[\hat{z}_n]$, which requires order
$m^2\log(m)$ operations because $z_n$ is an $m\times m$ array (or a
vector of length $m^2$).
With $z_n$ computed, we evaluate $V_y(x,\omega \cdot x+z_n)$, requiring
$m^2$ operations, and then we transform it with FFT, again
costing $\mathcal{O}(m^2\log(m))$ operations.

To compute $\hat{z}_{n+1}$, we now only need to perform component-wise
multiplication of the arrays in (\ref{NumEq}), requiring
$\mathcal{O} (m^2)$ operations. That is, we multiply the $(i,j)$-component of
$\mbox{fft}[V_y(x,\omega \cdot x(i,j)+z_n(i,j))]$ by the $(i,j)$-component
of $\Delta t (\gamma+\xi_1(i)^2 +
\xi_2(j)^2)^{-\beta}$ for each $1 \leq i,j \leq m$.
We then add to it the $(i,j)$-component of $\hat{z}_n(i,j)$
and then divide by the $(i,j)$-component of
$1+\Delta t
(\gamma  + \xi_1(i)^2 + \xi_2(j)^2)^{1-\beta}-\gamma \Delta t
(\gamma + \xi_1(i)^2+\xi_2(j)^2)^{-\beta}$.

If $\hat{z}_n$ and $\mbox{fft}[V_y(x,\omega \cdot x+z_n)]$
were represented as vectors of length $m^2$, then this same procedure
would amount to multiplying $\mbox{fft}[V_y(x,\omega \cdot x+z_n)]$
by the $m^2\times m^2$ diagonal matrix representing
$\Delta t (\gamma +\xi_1^2 + \xi_2^2)^{-\beta}$, adding $\hat{z}_n$, and
then multiplying the result by the $m^2\times m^2$ diagonal matrix
representing
$(1+\Delta t (\gamma  + \xi_1^2 + \xi_2^2)^{1-\beta}-\gamma \Delta t
(\gamma + \xi_1^2+\xi_2^2)^{-\beta})^{-1}$.
\end{remark}
\begin{remark}
  This method would also work in a setting
where the gradient part of the energy functional $S_{\ep}$
is replaced by the fractional laplacian. We could
use all of the same techniques above for solving the gradient descent for
\begin{equation} \label{EnFrac}
  S_{\ep}^{\delta}(z) = \int_{\mathbb{T}^d} z \, (-\Delta)^{\delta}z  + \ep V(x,\omega \cdot x+z) dx,
\quad \delta \in (0,1).
\end{equation}
The
critical points of (\ref{EnFrac}) will solve the Euler-Lagrange equation
\[
-(-\Delta)^{\delta}z  = \ep V_y(x,\omega \cdot x + z).
\]
Using the metric on $H^{\beta \delta}$ given by
$ \langle v ,w \rangle_{H^{\beta \delta}}=
\langle (\gamma + (-\Delta)^{\delta})^{\beta}v, w \rangle_{L^2}$,
we arrive at the descent equation
\begin{equation} \label{num_frac}
\partial_t z = -(\gamma +  (-\Delta)^{\delta})^{1-\beta}z +
(\gamma +  (-\Delta)^{\delta})^{-\beta}(\gamma z - V_y(x,\omega \cdot x+z)).
\end{equation}

Because $z$ is periodic and the operator $(-\Delta)^{\delta}$
is diagonal in the Fourier coefficients,
implementing (\ref{num_frac})
numerically is the same as described above, except with
powers of $|\xi |^{2\delta}$ in place of $|\xi |^2$.
\end{remark}
\section{Perturbation method for computing minimizers and
the jumps in $DA_{\varepsilon}(\omega)$} \label{Pert}

\subsection{Foliations and laminations of minimizers}
Following \cite{Senn95},
we define
$\Gamma_{\omega} = \{(k,l)\in \mathbb{Z}^d \times \mathbb{Z} : \bar{\omega} \cdot (k,l) = 0\}$,
where $\bar{\omega} = (\omega,-1)$,
and $\mathcal{M}(\bar{\omega}) \subset \mathcal{M}_{\omega}$, the set of maximally periodic $u\in \mathcal{M}_{\omega}$ by
\[
\mathcal{M}(\bar{\omega}) = \left\{u\in \mathcal{M}_{\omega} : u(x+k)+l =u(x), \forall (k,l)\in
\Gamma_{\omega} \right\}.
\]
For each $\omega\in \mathbb{R}^d$ the set
$\mathcal{M}(\bar{\omega})$ is closed and totally ordered.
The closedness is a classical argument (see \cite{Morse}).
The total order, meaning that if
$u$, $v \in \mathcal{M}(\bar{\omega})$ then either $u>v$ or $u<v$ or $u \equiv v$, is a
consequence of the maximum principle for elliptic partial differential
equations. Let $x = (x_1, \ldots , x_d)$, and we write $(x,x_{d+1}) \in \mathbb{R}^{d+1}$.
Let $\omega\in \mathbb{R}^d$, then the total order of $\mathcal{M}(\bar{\omega})$ means that for a given
$(x,x_{d+1}) \in \mathbb{R}^{d+1}$ there is at most one $u \in \mathcal{M}(\bar{\omega})$ such that
 $x_{d+1}=u(x)$. That is, each point in $\mathbb{R}^{d+1}$ belongs to the
graph of at most one $u \in \mathcal{M}(\bar{\omega})$. If for all $(x,x_{d+1}) \in \mathbb{R}^{d+1}$
there exists $u \in \mathcal{M}(\bar{\omega})$ with $x_{d+1}=u(x)$, then we say
$\mathcal{M}(\bar{\omega})$ is a foliation of $\mathbb{R}^{d+1}$. Because of the non-selfintersection
property (\ref{birkhoff}), such a foliation projects into $\mathbb{T}^{d+1}$.
It can happen that there are points $(x,x_{d+1}) \in \mathbb{R}^{d+1}$ for
which there does not exist any $u \in \mathcal{M}(\bar{\omega})$ with $x_{d+1}=u(x)$. In this
case we say $\mathcal{M}(\bar{\omega})$ is a lamination of $\mathbb{R}^{d+1}$ (and projects to
a lamination of $\mathbb{T}^{d+1}$). For this reason, a lamination is
sometimes referred to as a  ``foliation with gaps''.

If $\bar{\omega}$ is rationally dependent, and if $\mathcal{M}(\bar{\omega})$ defines a
lamination, then there are minimizers $u \in \mathcal{M}_{\omega}$ whose graphs
lie in the gaps of $\mathcal{M}(\bar{\omega})$, \cite{Bangert89}.
In addition, if we choose a direction
\[
\beta \in \mbox{span}_{\mathbb{R}}\{ k \in \mathbb{Z}^d : \omega\cdot k \in \mathbb{Z}\} \cap S^{d-1}
\]
then there is a $u \in \mathcal{M}_{\omega}$ such that $u$ is asymptotic to some
$u^+ \in \mathcal{M}(\bar{\omega})$ in the direction $\beta$ and asymptotic to some
$u^- \in \mathcal{M}(\bar{\omega}) $ in the direction $-\beta$. For details of the
asymptotic behavior, see \cite{Senn95}, page 350. Such a $u$ is
said to be \emph{heteroclinic} in the direction $\beta$. A formula
for the one-sided directional derivative of $A_{\varepsilon}$ at
a point $\omega$ is given on
page 356 of \cite{Senn95}, and involves integrating the action over the
gaps defined by the elements of $\mathcal{M}(\bar{\omega})$ and the heteroclinics between them.
We will use perturbation methods to calculate the gap borders
and the heteroclinics lying in the gaps.

\subsection{Lindstedt series for solutions}
We seek \emph{plane-like} solutions $u_{\ep}(x)$ of
\begin{equation} \label{MainEq}
\Delta u = \ep V_y(x,u)
\end{equation}
that can be expanded as $u_{\ep}(x) = u_0 + \ep u_1 + \ep^2 u_2 + \ldots$.
The series $\sum_{j\geq 0}\ep^ju_j$ shall be referred to as the
\emph{Lindstedt series} for the solution $u_{\ep}$.
Substituting the series into equation (\ref{MainEq}) and matching powers
we arrive at the following equations for each order of $\ep$
\begin{equation}\label{eqnList} \begin{split}
\ep^0: \quad \Delta u_0 & = 0 \\
\ep^1: \quad \Delta u_1 & = V_y (x,u_0)\\
\ep^2: \quad \Delta u_2 & = V_{yy}(x,u_0)u_1(x)\\
\ep^3: \quad \Delta u_3 & = V_{yy}(x,u_0)u_2(x) +
\frac{1}{2}V_{yyy}(x,u_0)u_1^2(x)\\
\ep^4: \quad \Delta u_4 & = V_{yy}(x,u_0)u_3(x) + V_{yyy}(x,u_0)u_1(x)u_2(x)
+\frac{1}{3!}V_{yyyy}(x,u_0)u_1^3(x)\\
&\vdots
\end{split}\end{equation}
In order that $u_{\ep}$ be a plane-like solution, we require $u_0$ to be
affine and $u_j$ be periodic for each $j>1$.
Using the notation $[\cdot]_j$ to refer to the $j$-th coefficient of the
power series in $\ep$, we write
\begin{displaymath}
V_y(x,u_{\ep}) = V_y(x,u_0+\ep u_1 + \ldots)= \sum_{j\geq 0}
[V_y(x,u_{\ep})]_{j}\ep^j.
\end{displaymath}
We will also write $u^{<j}$ for the first $j$ terms in the Lindstedt series:
$u^{<j} = u_0 + \ldots + \ep^{j-1}u_{j-1}$.
The $j$-th order equation in the list (\ref{eqnList}) has the form
\begin{equation} \label{Lind}
\ep^j: \quad \Delta u_j = [V_y(x,u)]_{j-1} =  [V_y(x,u^{<j})]_{j-1}.
\end{equation}

The zeroth order equation is satisfied by \emph{any} affine function, so
we take $u_0 = \omega \cdot x + \alpha$, and at this point we are
free to choose $\alpha$ as we like.
To solve the $j$-th order equation, we must have that
$\int_{\mathbb{T}^d} [V_y(x,u^{<j})]_{j-1}dx=0$. This compatibility condition is
what forces specific choices of $\alpha$ once $\omega$ has been fixed.

We also note that the solution of (\ref{Lind}) is determined
only up to an additive constant, which will be chosen so that the
equation of the following step has a solution. That is,
the average of $u_j$ is chosen so that equation for $u_{j+1}$
is solvable.

\subsection{Existence of the Lindstedt series to all orders} \label{exist}
We consider $\omega\in \frac{1}{N}\mathbb{Z}^d$ fixed, and seek $u_{\ep}$ solving
(\ref{MainEq}) such that $u_{\ep} (x)- \omega \cdot x \in L^{\infty}(N\mathbb{T}^d)$.
We must have $u_0(x) = \omega \cdot x +\alpha$ to solve $\Delta u_0 = 0$. The
choice of $\alpha$ is free, and each value of $\alpha \in [0,1)$
will result in a different set of equations for the $u_j$ with $j \geq 1$.
In this section, we show that there are at least two choices of
$\alpha \in [0,1)$
such that (\ref{Lind}) has a solution for each $j \geq 0$.
\begin{lemma} \label{lem1}
For fixed $\omega\in \frac{1}{N}\mathbb{Z}^d$, there are at least two choices of
$\alpha \in [0,1)$ such that
$\int_{N\mathbb{T}^d}V_y(x,\omega \cdot x + \alpha) \, dx = 0$.
\end{lemma}
\begin{proof}
Define $\Phi_1 : \mathbb{R} \to \mathbb{R}$ by
$\Phi_1 (\alpha) = \int_{N\mathbb{T}^d}V_y(x,\omega \cdot x + \alpha) \, dx$.  Then $\Phi_1$ is
continuous, and we have
\begin{align*}
\int_0^1 \Phi_1 (\alpha) \, d\alpha & =
\int_0^1\int_{N\mathbb{T}^d}V_y(x,\omega \cdot x + \alpha) \, dx d\alpha \\
& = \int_{N\mathbb{T}^d}\int_0^1\frac{\partial \phantom{u}}{\partial \alpha}
V(x,\omega \cdot x + \alpha) \, d\alpha dx \\
& = \int_{N\mathbb{T}^d} V(x,\omega \cdot x ) - V(x,\omega \cdot x + 1) \, dx = 0.
\end{align*}
Thus, $\Phi_1$ must have a zero in $[0,1)$, and since
$\Phi_1 (\alpha + 1) = \Phi_1 (\alpha)$, by the periodicity of $V$,
we know that $\Phi_1$ must have at least two zeros.
\end{proof}

For any such choice
of $\alpha$, $\int_{N\mathbb{T}^d}V_y(x,u_0) \, dx = 0$, and
therefore there exists a family of periodic solutions,
$u_1(x) = u_1^*(x) + \lambda$, of
$\Delta u_1 = V_y(x,u_0)$, differing only by an additive constant.
We will write $u_1^*$ for the member of the family with average zero.
\begin{theorem}
Let $u_0 = \omega \cdot x + \alpha$, with $\omega\in \frac{1}{N}\mathbb{Z}^d$ and $\Phi_1 (\alpha)=0$. If
\begin{equation} \label{twist}
\int_{N\mathbb{T}^d} V_{yy}(x, u_0) \, dx \ne 0
\end{equation}
then each equation
$\Delta u_j = \left[ V_y(x,u^{<j} ) \right]_{j-1}$ has a solution for
all $j \geq 1$.
\end{theorem}
\begin{proof}
From Lemma \ref{lem1}, we have a family of solutions
$u_1 = u_1^* + \lambda$, of $\Delta u_1 = V_y(x,u_0)$.
If we set
\[
\lambda^{(1)} = -\frac{\int_{N\mathbb{T}^d}V_{yy}(x,u_0)u_1^* \,
dx }{\int_{N\mathbb{T}^d}V_{yy}(x,u_0) \, dx}
\]
then, $\int_{N\mathbb{T}^d} \left[ V_y(x,u^{<2} ) \right]_{1}\, dx =
\int_{N\mathbb{T}^d}V_{yy}(x,u_0)(u_1^*+\lambda ^{(1)}) \, dx =0$, so the equation
$\Delta u_2 = V_{yy}(x,u_0)u_1$ is solvable for $u_2 = u_2^* + \lambda$
when we choose $u_1(x) = u_1^*(x) + \lambda ^{(1)}$.
To continue inductively, we note that for each $j \geq 1$,
\[
\left[ V_y(x,u^{<j} ) \right]_{j-1} = V_{yy}(x,u_0)u_{j-1} + R_j(u_1,
\ldots,u_{j-2}).
\]

Suppose that we have solutions of (\ref{Lind}) for $j=1, \ldots , n$,
and the
constants $\lambda ^{(1)}, \ldots, \lambda ^{(n-1)}$ have been selected
so that $u_j(x) = u_j^*(x)+\lambda ^{(j)}$.
We choose
\[
\lambda ^{(n)} =  -\frac{\int_{N\mathbb{T}^d}V_{yy}(x,u_0)u_n^* + R_n(u_1,\ldots ,
 u_{n-1})\, dx  }{\int_{N\mathbb{T}^d}V_{yy}(x,u_0) \, dx}
\]
so that
\[
\int_{N\mathbb{T}^d} \left[ V_y(x,u^{<n+1} ) \right]_{n} dx =
\int_{N\mathbb{T}^d} V_{yy}(x,u_0)(u_n^*+\lambda ^{(n)}) + R_n(u_1,\ldots,u_{n-1}) dx = 0.
\]
Thus $\Delta u_{n+1} = \left[ V_y(x,u^{<n+1} ) \right]_{n}$ has a family of
solutions, $u_{n+1}(x)= u_{n+1}^*(x)+\lambda$, completing the induction.
\end{proof}
\subsection{Convergence of the Lindstedt series} \label{convergence}
We use a Newton
method to produce a sequence of functions $U_n(x,\ep)$ that are
analytic in $\ep$, and converge uniformly to a solution of
$-\Delta u_{\ep} + \ep V_y(x,u_{\ep})=0$ for
$\ep \in B_{\delta}(0) \subset \mathbb{C}$ for small enough $\delta >0$.
Thus, we produce
an $\ep$-analytic function, $u_{\ep}$, solving (\ref{MainEq})
for $\ep \in B_{\delta}(0)$, so the Taylor series of $u_{\ep}$
must coincide with the Lindstedt series, proving
the convergence of the Lindstedt series. For similar convergence results,
 but for the case of
Diophantine frequencies, see \cite{Chierchia}.

\begin{lemma} \label{lem:analytic}
  Let $u_j\in H^{m+2}(\mathbb{T}^d)$ and $\ep \in \mathbb{C}$ and define
$u(\ep,x) = \sum_{j\in \mathbb{N}}\ep^j u_j(x)$, where
the series is convergent in $H^{m+2}$ for $|\ep|<r$, for some $r>0$
and  $m>d/2$.
Define $F:\mathbb{C}\to H^{m}(\mathbb{T}^d)$ by
\[
F(\ep;u) = \Delta u(\ep,x) + \ep V(x,u(\ep,x))
\]
Then the derivative of $F$ with respect to $\ep$ is
\[
D_{\ep} F(\ep;u) = \Delta D_{\ep} u(\ep,x) + V(x,u(\ep,x)) + \ep
V_y(x,u(\ep,x))D_{\ep} u(\ep,x).
\]
\end{lemma}
\begin{proof}
We define
\begin{align*}
G(\ep ;u) &= \Delta D_{\ep} u(\ep,x) + V(x,u(\ep,x)) + \ep
V_y(x,u(\ep,x))D_{\ep} u(\ep,x)  \\
H(\ep ;u) &= \Delta D_{\ep}^2u(\ep ,x) + 2V_y(x,u(\ep,x))D_{\ep} u(\ep,x)\\
& \qquad + \ep V_y(x,u(\ep,x))D_{\ep}^2u(\ep,x)
 + \, \ep V_{yy}(x,u(\ep,x))
\left( D_{\ep} u(\ep,x)\right)^2.
\end{align*}
We have
\begin{align*}
  & \int_0^1 H(\ep + \sigma \tau h;u)\tau h \, d\sigma = G(\ep + \tau h;u)
- G(\ep ;u) \\
  & \int_0^1 G(\ep + \tau h;u) h \, d\tau = F(\ep + h;u) - F(\ep ;u)
\end{align*}
so that
\begin{align*}
  F(\ep + h;u) - F(\ep ;u) -G(\ep;u)h =  \int_0^1 \int_0^1 H(\ep + \sigma
 \tau h;u)\tau h^2 \, d\sigma \, d\tau
\end{align*}
and
\begin{align*}
  \left\| H(\ep;u) \right\|_{H^m} & \leq \left\| D_{\ep}^2u(\ep ,x)\right\|_{H^{m+2}}
+ 2 \left| V_y \right|_{C^0} \left\| D_{\ep} u(\ep,x) \right\|_{H^m} \\
&\qquad + |\ep| \left| V_y \right|_{C^0} \left\| D_{\ep}^2u(\ep,x)  \right\|_{H^m}  +
|\ep| \left|V_{yy}\right|_{C^0} \left\| \left( D_{\ep} u(\ep,x)\right)^2  \right\|_{H^m}.
\end{align*}
From the Gagliardo-Nirenberg inequality \cite{Nirenberg59}, we have that
$\left( D_{\ep} u(\ep,x)\right)^2 \in H^m$ because $m>d/2$,
and that there is a constant depending on $m$ and $d$, such that
 $\left\| \left( D_{\ep} u(\ep,x)\right)^2  \right\|_{H^m} \leq C(m,d)
\left\| D_{\ep} u(\ep,x) \right\|_{H^m}^2$. Hence,
\begin{align*}
  \left\|F(\ep + h;u) - F(\ep ;u)-G(\ep;u)h\right\|_{H^m} &\leq \int_0^1 \int_0^1
\left\|H(\ep + \sigma  \tau h;u)\right\|_{H^m}\tau |h|^2 \, d\sigma \, d\tau \\
& \leq C\left( m,d,\|u\|_{m+2},|V|_{C^2},\ep\right)|h|^2.
\end{align*}
Thus, $F$ is differentiable and $D_{\ep} F(\ep;u) = G(\ep;u)$
\end{proof}
We define $F_{\ep}(u) = -\Delta u + \ep V_y(x,u)$.
It was shown in the Section \ref{exist} that for any
fixed $M\in \mathbb{N}$, we can solve the first $M$ equations from
(\ref{eqnList}) for $u^{<M}_{\ep}$,
and that $u^{<M}_{\ep}$ will solve (\ref{MainEq})
up to order $\ep^{M}$.
That is, $F_{\ep}(u^{< M}_{\ep}) = \mathcal{O} (\ep^{M})$.
We set $U_0 = u_{\ep}^{< M}$ for a sufficiently
large $M$ to be determined later, and for $j \geq 1$ we define
\begin{equation} \label{NewtonMethod}
U_{n+1} = U_n - DF_{\ep}(U_n)^{-1} F(U_n).
\end{equation}

Let $m>d/2$ be fixed, and note that for any $M>0$, $u^{<M}_{\ep} \in H^m$ by
the regularity theory for elliptic PDEs.
If $U_n(x,\ep)$ is analytic in $\ep$ and is $H^{m+2}$ in $x$, then
$F_{\ep}(U_n) = -\Delta U_n + \ep V_y(x,U_n)$ is analytic in $\ep$ and $H^m$
in $x$ by the result in Lemma \ref{lem:analytic}.
We have $DF_{\ep}(U_n) \eta = -\Delta \eta +\ep V_{yy}(x,U_n)\eta$, and we
need to consider carefully the behavior of $DF_{\ep}(U_n)^{-1}$.
To simplify notation, we define $L_{\ep}^n : H^{m+2}(N\mathbb{T}^d) \to H^m(N\mathbb{T}^d)$ as
\[
L_{\ep}^n = DF_{\ep}(U_n) = -\Delta +\ep V_{yy}(x,U_n).
\]
$L_{\ep}^n$ is a small
perturbation of $-\Delta : H^{m+2} \to H^m$. Now, $-\Delta$ maps the
codimension one subspace $H^{m+2}/\mathbb{R}$ of its domain to the codimension one
subspace $H^m/\mathbb{R}$ of its range in a bounded, invertible
way. But it has the simple eigenvalue $\lambda = 0$,
with eigenspace spanned by the constant functions.

\begin{lemma} \label{bound}
Let $P_0:H^m \to H^m/\mathbb{R}$ be the orthogonal projection onto $H^m$ functions
with zero average. That is, if $f\in H^m$, then
$P_0f = f - \fint_{N\mathbb{T}^d} f\, dx$. Let $-\Delta_0:H^{m+2}/\mathbb{R} \to H^m/\mathbb{R}$ be
the restriction of the laplacian, and let $Y$ be the image of  $H^{m+2}/\mathbb{R}$ under
$-\Delta_0 + \ep V_{yy}(x,U_n)$.
Suppose there are constants $q, c_1, \delta_0 > 0$ such that
\begin{equation}\label{degen}
\langle \ep V_{yy}(x,U_n), g \rangle_{H^m} \geq c_1|\ep |^q,
\quad \forall \, g \in Y^{\perp}, \, \|g\|_{H^m}=1
\end{equation}
for $|\ep| < \delta_0$.
Then $\| (L_{\ep}^n)^{-1} \|_{\mathcal{L}(H^m)} \leq c\ep^{-q}$, for some $c>0$.
\end{lemma}
\begin{proof}
Let $Q = -\Delta_0 + P_0\ep V_{yy}(x,U_n)$, then
 $Q:H^{m+2}/\mathbb{R} \to H^m/\mathbb{R}$, and for small $\ep$ will have
a bounded inverse, $Q^{-1}:H^m/\mathbb{R} \to H^{m+2}/\mathbb{R}$.

We have that
$-\Delta_0 + \ep V_{yy}(x,U_n)= -\Delta_0 + P_0\ep V_{yy}(x,U_n)
+P_0^{\perp}\ep V_{yy}(x,U_n)$ maps $H^{m+2}/\mathbb{R}$ into $H^m$ and is a small
perturbation of $Q$, and there is a $c_2 >0$ such that
$\|-\Delta_0 \eta+ \ep V_{yy}(x,U_n)\eta\|_{H^m} \geq c_2\|\eta\|_{H^{m+2}}$
for all $\eta \in H^{m+2}/\mathbb{R}$.
Thus, $Y$ is a codimension one linear space
isomorphic to $H^m/\mathbb{R}$ lying inside $H^m$, and $Y^{\perp} = \{g \in H^m :
\int_{N\mathbb{T}^d}(-\Delta_0 \eta + \ep V_{yy}(x,U_n)\eta)g\, dx =0,\forall
\eta \in H^{m+2}/\mathbb{R} \} $ is one dimensional.

The condition (\ref{degen})
implies that the image of constant functions under $L_{\ep}^n$ does
not lie entirely in $Y$, but has a component in the $Y^{\perp}$
direction. Thus the image of $H^{m+2}$ under $L_{\ep}^n$ is $H^m$, and
for small $\ep$, $L_{\ep}^n$ will be invertible.

Let $P_Y$ and $P_Y^{\perp}$
denote the orthogonal projections of $H^m$ onto $Y$ and $Y^{\perp}$.
We let $g \in Y^{\perp}$ be a unit vector such that
$P_Y^{\perp}\xi = \langle \xi, g \rangle_{H^m} g$ for $\xi \in H^m$,
which is possible because $Y^{\perp}$ is one dimensional.
Let $\xi \in H^m$, so $\xi = L_{\ep}^n \eta$ for some
$\eta \in H^{m+2} $. We write $\eta = \eta_1+ \eta_0$ with
$\eta_1 \in H^{m+2}/\mathbb{R} $,
and $ \eta_0 \in \mathbb{R}$, and  $\xi = \xi_Y + \xi_{\perp}$
with  $\xi_y \in Y$ and $\xi_{\perp} \in Y^{\perp}$. Writing
$\xi$ in terms of $\eta_1 , \, \eta_0$ we have
\begin{align*}
 \Delta \eta + \ep V_{yy}(x,U_n)\eta &=
 \Delta_0 \eta_1 + \ep V_{yy}(x,U_n)\eta_1 + \ep V_{yy}(x,U_n)\eta_0 \\
 & =\Delta_0 \eta_1 + \ep V_{yy}(x,U_n)\eta_1 + P_Y\ep V_{yy}(x,U_n)\eta_0 +
P_Y^{\perp}\ep V_{yy}(x,U_n)\eta_0 .
\end{align*}
The term $\Delta_0 \eta_1 + \ep V_{yy}(x,U_n)\eta_1 +
P_Y\ep V_{yy}(x,U_n)\eta_0 \in Y$, so the component of $L_{\ep}^n \eta$
in $Y^{\perp}$ is $\xi_{\perp}=\langle L_{\ep}^n \eta ,g\rangle_{H^m} g
= \langle \ep V_{yy}(x,U_n)\eta_0 , g \rangle_{H^m} g $.
Thus, $\eta_0 = \xi_{\perp}(\langle \ep V_{yy}(x,U_n) , g \rangle_{H^m})^{-1}$
and $\eta_1$ is given by
\[
\eta_1 = \left(\Delta_0 + \ep V_{yy}(x,U_n) \right)^{-1}
\left( \xi_Y + \frac{P_Y\ep V_{yy}(x,U_n)\xi_{\perp}}{\langle
\ep V_{yy}(x,U_n) , g \rangle_{H^m}} \right).
\]
Hence, $\|(L_{\ep}^n)^{-1}\xi\|_{H^{m+2}}\leq \frac{1}{c_1|\ep|^q}\|\xi_{\perp}\|_{H^m}
+ \frac{1}{c_2}\left(\|\xi_Y \|_{H^m}+ \ep\|V_{yy}\|_{H^m}\frac{1}{c_1|\ep|^q}
\|\xi_{\perp}\|_{H^m} \right) \leq c |\ep|^{-q}\|\xi \|_{H^m}$.
\end{proof}
Recall that $ (L_{\ep}^n)^{-1}$ is a compact operator
by the regularity theory for elliptic PDE. In particular, the
eigenvalues of $L_{\ep}^n$ are isolated from the spectrum, and if $\lambda_n$
is an eigenvalue of $L_{\ep}^n$, then
the resolvent of $L_{\ep}^n$, written as
$R(L_{\ep}^n,\zeta) = (L_{\ep}^n-\zeta I)^{-1}$, has the representation
\begin{equation}\label{spectral}
R(L_{\ep}^n,\zeta) =
\sum_{j=0}^{\infty}(\lambda_n-\zeta)^jQ_n^{j+1} + \frac{1}{\lambda_n-\zeta}P_n,
\end{equation}
where $Q_n$ is bounded, and
$P_n$ is the spectral projection on the the $\lambda_n$ eigenspace:
\[
P_n = -\frac{1}{2\pi i}\int_{\Gamma}R(\zeta,L_{\ep}^n) \, d\zeta,
\]
and $\Gamma$ is a closed curve enclosing $\lambda_n$ but no other point of
the spectrum (see \cite{Kato66}). The principal eigenvalue, $\lambda_0(\ep)$,
of $L_{\ep}^n$ is simple
because $L_{\ep}^n$ is an elliptic operator, \cite{Evans98}. This means that
$\lambda_0(\ep)$ is analytic in a neighborhood of $\ep = 0$, \cite{Kato66}.
In the iteration process, $(L_{\ep}^n)^{-1} = R(L_{\ep}^n,\zeta) =
\sum_{j=0}^{\infty}\lambda_n^jQ_n^{j+1} + \frac{1}{\lambda_n}P_n,$ will act on
$F_{\ep}(U_n)$. At each step, we need the function $U_n$ to be analytic
in $\ep$, so we need $(L_{\ep}^n)^{-1}F_{\ep}(U_n)$ to be analytic.
\begin{proposition} \label{prop:convergence}
Assuming condition (\ref{degen}) holds, there is a choice
of $M \in \mathbb{N}$ such that, if $U_0 = u^{<M}$, then $U_n$
will be analytic in $\ep$ for all $n \geq 0$, and the
Newton method (\ref{NewtonMethod}) converges uniformly in $\ep$ in a
neighborhood of $\ep = 0$.
\end{proposition}
\begin{proof}
In the iteration process, $(L_{\ep}^n)^{-1} = R(L_{\ep}^n,\zeta) =
\sum_{j=0}^{\infty}\lambda_n^jQ_n^{j+1} + \frac{1}{\lambda_n}P_n,$ will act on
$F_{\ep}(U_n)$.
As explained in the previous paragraph, $\lambda_0(\ep)$ is analytic
in a neighborhood of $\ep=0$, and will have a zero
of order $p \in \mathbb{N}$ at $\ep=0$. From the result of
Lemma \ref{bound} we may assume  $p < q$.

We will use induction on $n$ to show the analyticity of $U_{n+1}$.
$F_{\ep}(u^{<M})$ has a zero of order $M$ at $\ep = 0$ by construction
of $u^{<M}$, and we take $M>2q$.
By the expansion in (\ref{spectral}) and the result from Lemma \ref{lem:analytic},
for $\zeta = 0$ and
$\lambda_n = \lambda_0(\ep)$, $(L_{\ep}^0)^{-1}F_{\ep}(u^{<M})$
has
a zero of order $M-p$ if $M>p$ or a pole of order $p-M$
if $p>M$. Thus, taking $M>2q$ is more than enough to ensure $U_1=
u^{<M}+(L_{\ep}^0)^{-1}F_{\ep}(u^{<M})$ is
analytic in $\ep$, since $q > m$.
From (\ref{NewtonMethod}),
$F_{\ep}(U_1) = D^2F_{\ep}(U_0)(DF_{\ep}(U_0)^{-1}F_{\ep}(U_0))^2 + \ldots $,
and since $D^2F_{\ep}(U_0) = \ep V_{yyy}(x,U_0)$ is bounded,
we have $\| F_{\ep}(U_1)\|_{H^m} \leq C |\ep |
(\| (L_{\ep}^0)^{-1} \|_{\mathcal{L}(H^m)}\|F_{\ep}(U_{0})\|_{H^m})^2$.
Hence, $\| F_{\ep}(U_1)\|_{H^m} \leq C |\ep |^{1-2q+2M}$ by Lemma \ref{bound} and because
$\|F_{\ep}(U_{0})\|_{H^m} \leq |\ep |^M$. We chose $M$ such that $M-q > q$,
so  $\| F_{\ep}(U_1)\|_{H^m} \leq C |\ep |^{M_1}$ where $M_1 = 1+ 2(M-q) > 2q$.
We also have $\|U_1-U_0\|_{H^m} \leq c |\ep |^q$, so
\[
L_{\ep}^1 = -\Delta + \ep V_{yy}(U_1) =
-\Delta + \ep V_{yy}(U_0 + U_1-U_0) = L_{\ep}^0 + \mathcal{O} (\ep^q ).
\]
Hence, $L_{\ep}^1$ is a small perturbation of $L_{\ep}^0$, and
$\lambda_1(\ep)$ will have a zero of order $p$.
Thus, we have established the first step of the induction on $n$,
because $U_1$ is analytic in $\ep$, $F_{\ep}(U_1)$ has a zero of order $M_1 > 2q$,
and the principal eigenvalue has a zero of order $p <  q$ at $\ep =0$.

Now assume that $U_n$ is analytic in $\ep$ and $F_{\ep}(U_n)$ has a zero
of order $M_n >2q$ at $\ep = 0$.
$F_{\ep}(U_{n+1}) = D^2F_{\ep}(U_n)(DF_{\ep}(U_n)^{-1}F_{\ep}(U_n))^2 + \ldots $
and $\|D^2F_{\ep}(U_n)\|_{\mathcal{L}(H^m)}\leq \ep \| V_{yyy}\|_{L^{\infty}}$ is
independent of $n$. So
\begin{equation} \label{Un}
\| F_{\ep}(U_{n+1})\|_{H^m} \leq
C \ep \|(DF_{\ep}(U_n)^{-1}F_{\ep}(U_n))\|_{H^m}^2 \leq C \ep
(\ep^{M_n-q})^2.
\end{equation}
From (\ref{Un}) we have that $U_{n+1}$ is analytic, and $F_{\ep}(U_{n+1})$
has a zero of order at least $2q$ at $\ep = 0$. Just as for
the $n=0$ case, $\|U_{n+1}-U_{n}\|_{H^m} \leq c |\ep |^q$ so
$\lambda_{n+1}(\ep)$ has a zero of order $p <  q$ at $\ep =0$,
and the induction is complete.

To prove uniform convergence, we have from (\ref{Un})
the recurrence formula:
$$\| F_{\ep}(U_{n})\|_{H^m} <  (\ep^{-q}\|F_{\ep}(U_{n-1})\|_{H^m})^2,$$
provided we choose $\ep$ small enough (independently of $n$)
so that $C\ep <1$. Therefore, $\| F_{\ep}(U_{n})\|_{H^m} <
| \ep |^{-2q(2^n-1)}\|F_{\ep}(U_{0})\|_{H^m}^{2^n} \leq c|\ep |^{2^n(M-2q)+2q}$.
Hence, with $M>2q$, the error on
$F_{\ep}(U_{n})$ is bounded by $c \ep^{2^n}$, $c$ independent of $n$.
\end{proof}

\subsection{Connecting orbits and corners of the energy}
The connecting orbits, or heteroclinic orbits, exist
only when $\mathcal{M}(\bar{\omega})$ fails to produce a foliation of $\mathbb{T}^{d+1}$,
as described at the beginning of Section \ref{Pert}. Each
$u \in \mathcal{M}(\bar{\omega})$ satisfies $\Delta u = \ep V_y(x,u)$ because
it is a minimizer of $S_{\ep}$. The order-zero approximation
to $u$ has the form $u_0 = \omega \cdot x + \alpha$, for $\alpha \in [0,1)$.
This is a continuous family, whose graphs foliate $\mathbb{T}^{d+1}$.

From Lemma \ref{lem1} we know
that if $V$ satisfies the twist condition
(\ref{twist}) then there are at least two choices
of $\alpha \in [0,1)$ such that the Lindstedt equations (\ref{eqnList})
can be solved to all orders. Recall that the
function $\Phi_1$, defined in Lemma \ref{lem1}, has at least two zeros
in $[0,1)$. If the zero set of $\Phi_1$ is not all of $[0,1)$, then
a choice of $\alpha$ must be made in order to solve
$\Delta u_1 = V_y(x,\omega \cdot x+\alpha)$ for periodic $u_1$. Thus, the
order-$\ep$ approximations, given by $u_0 + \ep u_1$
are not a continuous family indexed by $\alpha \in [0,1)$, but
rather by a strict subset of $[0,1)$. The graphs of the functions in t
his family no longer foliate $\mathbb{T}^{d+1}$.

If $\Phi_1(\alpha) \equiv 0$, no gaps appear
in the approximation up to first order in $\ep$.
In this case, we can find periodic $u_1$ solving
$\Delta u_1 = V_y(x,\omega \cdot x+\alpha)$ for arbitrary $\alpha$, and
the order-$\ep$ approximations do foliate $\mathbb{T}^{d+1}$.
We can then move
on to try to solve $\Delta u_2 = [V_y(x,u^{<2})]_{1}$ for periodic $u_2$.
With $\alpha$ free, we define $\Phi_2 : [0,1) \to \mathbb{R}$ by
$ \Phi_2(\alpha) = \int_{N\mathbb{T}^d}  [V_y(x,u^{<2})]_{1} dx$.  Now either
$\Phi_2(\alpha) \equiv 0$ or there is some $\alpha \in [0,1)$ with
$\Phi_2(\alpha) \ne 0$. In the latter case, we must make a choice of
$\alpha$ such that $\Phi_2(\alpha) = 0$. In the former case, we move
on to find $u_3$. This can continue as long as
$ \Phi_j(\alpha) = \int_{N\mathbb{T}^d}  [V_y(x,u^{<j})]_{j-1} dx$ is identically
zero.
Otherwise, a choice of $\alpha$ must be fixed, and the foliation breaks down.

If at the $j$-th step we find that $\Phi_j(\alpha) \ne 0$ for some value of
$\alpha$ then we make the ansatz that the heteroclinic orbit will be of
the form
$u_h(x)  = \omega \cdot x + \alpha(\ep^{j/2}x) + \ep u_1(x) + \ldots$. And
$\alpha(\ep^{j/2}x) \to \alpha_{\pm}$ as $x \to \pm \infty$ where
$\alpha_{\pm}$
satisfy $\Phi_j(\alpha_{\pm}) = 0$. Then $\Delta \alpha = \mathcal{O} (\ep^j)$ and
the $j$-th order equation will be $\Delta \alpha + \Delta u_j =
[V_y(x,u^{<j})]_{j-1}$.
Thus, $\alpha$ will solve a PDE of the form $\Delta \alpha = f(x,\alpha)$
with boundary conditions
$\alpha(\ep^{j/2}x) \to \alpha_{\pm}$ as $x \to \pm \infty$,
where $\int_{N\mathbb{T}^d} f(x,\alpha) \, dx \ne 0$. That is, $f$ is the term that
keeps $[V_y(x,u^{<j})]_{j-1}$
from having a zero average for any $\alpha$. This is carried out in
detail in Section \ref{example}
for a specific example.

\subsection{Energies involving the fractional Laplacian}
The existence of minimizers in the case $\omega\in \frac{1}{N}\mathbb{Z}^d$
has been established for energy functionals involving
fractional powers of the laplacian, \cite{LlaveVal09}, \cite{BlassLlaveVal10}.
As described in the remark at the end of Section \ref{implement},
the Euler-Lagrange equation for the functional
\begin{equation*}
  S_{\! \ep}^{\delta}(u) = \int_{\mathbb{R}^d} u \, (-\Delta)^{\delta}u  + \ep V(x,u) dx,
\quad \delta \in (0,1), \quad u = \omega \cdot x+z(x),
\end{equation*}
is
\[
-(-\Delta )^{\delta}u=V_y(x,u), \quad u = \omega \cdot x+z(x).
\]
Much of what has been described so far regarding solutions
of $\Delta u = V_y(x,u)$ carries over to this case. However,
the analogous properties of the associated minimal average
energy $A_{\ep}^{\delta}$ that are presented for $\delta =1$
in \cite{Senn91} and \cite{Senn95} need to be established
if one desires to investigate the differentiability properties
of $A_{\ep}^{\delta}$ in this case of non-local energy. This remains
an interesting challenge.

\section{An example of the Lindstedt series} \label{example}
Consider the potential $V:\mathbb{R}^{d+1} \to \mathbb{R}$ given by
 $V(x,y) = \sin(2\pi k \cdot x)\cos(2\pi y)$, with $k \in \mathbb{Z}^d$.
 For a fixed $\omega\in \mathbb{R}^d$, We compute the
jumps in the derivative of the average energy functional, $A_{\varepsilon}(\omega)$
defined in (\ref{Ae}),
using asymptotic expansions of the connecting orbits.

Let $u_{\ep}(x)$ solve $\Delta u_{\ep} = \ep V(x,u_{\ep})$. Writing
$u_{\ep} = \sum_j \ep^j u^{(j)}$ we see that $\Delta u_0 = 0$, and thus
$u_0 = \omega \cdot x +\alpha$, $\omega \in \mathbb{R}^d$,
$\alpha \in \mathbb{R}$. We are considering $\omega$ fixed, but
the choice of $\alpha$ is free
and we have a family of solutions parametrized by $\alpha \in \mathbb{R}$.

However, in order to calculate $u_1$ we must solve $\Delta u_1 = V_y(x,u_0)$
for a periodic function $u_1$. The average of $V_y(x,u_0)$ depends on $\omega$,
so in this example we will choose $\omega= k$ so that the average of
$V_y(x,u_0)$
is nonzero for some choices of $\alpha$.
Thus, the requirement that $V_y(x,u_0) =
-2\pi\sin(2\pi k\cdot x)\sin(2\pi k\cdot x + 2\pi \alpha)$
has zero average forces $\cos(2\pi \alpha) = 0$. We say
that $u_{\ep}$ has a first order resonance if $\alpha$ is restricted
in solving for the order $\ep^1$ term in the Lindstedt series.

We can calculate the first two terms in the Lindstedt series for each
admissible value of $\alpha = 1/4, \, 3/4$. We have $\Delta u_1 =
\mp \pi \sin(4\pi kx)$ so that the two solutions are:
\[ \begin{split}
u_{\ep}^c &= kx + \frac{3}{4} -\frac{\ep}{16\pi |k|^2}
\sin(4\pi k\cdot x)+
O(\ep^2)\\
u_{\ep}^m &= kx + \frac{1}{4} + \frac{\ep}{16\pi |k|^2}
\sin(4\pi k \cdot x)
+ O(\ep^2).
\end{split} \]
The superscripts $c$ and $m$ indicate that the solution for
$\alpha = 1/4$
is a minimizer of the energy functional
\[
A_{\varepsilon}(u) = \lim_{N\to \infty}
\frac{1}{N^d} \int_{[0,N]^d}\frac{1}{2}|\nabla u|^2 + \ep V(x,u) \, dx,
\]
and the solution for $\alpha = 3/4$ is a critical point, but not a minimizer.
We restrict our attention to minimizers, where $A_{\varepsilon}$ can be considered
as a function of the rotation vector $\omega$.

To compute the jump in the gradient of $A_{\varepsilon}$ at $\omega = k$
using the formula in \cite{Senn95} we need to find
minimizers of $A_{\varepsilon}$ that are heteroclinic between $u_{\ep}^m$ and
$u_{\ep}^m+1$.
We search for an asymptotic solution to $\Delta u = \ep V_y(x,u)$ that
has the form $u_{\ep}^h(x) = kx + \alpha(\sqrt{\ep}x) +
\ep u_1(x) +
o(\ep)$.

The order $\ep^1$ equation is
\begin{equation}\label{eqn:alpha_u} \begin{split}
\Delta \alpha + \Delta u_1 &= -2\pi \sin(2 \pi kx)\sin(2\pi kx+2\pi
\alpha) \\
&=-\pi\cos(2\pi \alpha) + \pi \cos(2 \pi \alpha) \cos(4 \pi kx) - \pi
\sin(2\pi \alpha)\sin(4 \pi kx).
\end{split} \end{equation}
We want to choose $\alpha$ so that it is essentially
one-dimensional (i.e. for some $\eta \in \mathbb{R}^d$, we want $\alpha$ to be
a function of $\eta \cdot x$). It should also satisfy $\Delta \alpha = -\pi
\cos(2 \pi \alpha)$ and we can eliminate those two terms from the
equation above. At this point, any choice of $\eta$ is reasonable, and
in the end we will chose $\eta$ to be in the direction in which we want
to differentiate $A_{\varepsilon}(\omega)$ (thus, $\eta$ will typically be a standard
basis vector).

Letting $\hat{\eta}$ denote $\frac{1}{| \eta |}\eta$,
we have
\[
\alpha(z) = \frac{1}{\pi}\arctan(\sinh(\sqrt{2}\pi \hat{\eta}\cdot z))
+ \frac{3}{4}
\]
Recall that in the expression for $u$, $\alpha$ is evaluated at
$z=\sqrt{\ep}x$.
Expanding $\sin(2 \pi \alpha(\sqrt{\ep \ }x))$ and
$\cos(2 \pi \alpha(\sqrt{\ep \ }x))$
in Taylor series, for this choice of $\alpha$, we have
\[ \begin{split}
\sin(2 \pi \alpha(\sqrt{\ep \ }x)) &= -1 + 2(\sqrt{2\ep \ }\pi
\hat{\eta}\cdot x)^2 -
\frac{4}{3})(\sqrt{2\ep \ }\pi \hat{\eta}\cdot x)^4+ \ldots \\
&=-1 + O(\ep) \\
\cos(2 \pi \alpha(\sqrt{\ep \ }x)) &= 2(\sqrt{2\ep \ }\pi
\hat{\eta}\cdot x)
- \frac{5}{3}(\sqrt{2\ep \ }\pi \hat{\eta}\cdot x)^3 + \ldots \\
&= O(\sqrt{\ep \ }).
\end{split} \]
Then equation (\ref{eqn:alpha_u}) becomes $\Delta u_1 = \pi \sin(4\pi kx)  +
O(\sqrt{\ep \ })$
and we get for $u_{\ep}^h$ the expression:
\[ \begin{split}
u_{\ep}^h(x) &= k\cdot x + \alpha(\sqrt{\ep \ }x) + \ep u_1(x) +
O(\ep^{3/2}) \\
& = k \cdot x + \frac{1}{\pi}\arctan(\sinh(\sqrt{2\ep \ }\pi
\hat{\eta}\cdot x)) + \frac{3}{4}
- \frac{\ep}{16 \pi |k|^2}\sin(4\pi kx) + o(\ep).
\end{split}\]

To be precise, we should write $u_{\ep,\eta}^h$ to show the dependence of
$u_{\ep}^h$ on $\eta$. It is important to notice that $u_{\ep,\eta}^h$ is
heteroclinic from $u_{\ep}^m+1$ to $u_{\ep}^m$ in the direction $-\eta$.
Similarly, $u_{\ep,-\eta}^h$ is heteroclinic from $u_{\ep}^m+1$ to $u_{\ep}^m$
in the direction $\eta$, and from $u_{\ep}^m$ to $u_{\ep}^m+1$
in the direction
$-\eta$.

\subsection{Computing the gradient of $A_{\varepsilon}(\omega)$}

Here we restrict $\eta$ to be a standard basis vector $e_j$. To ease notation
we will write $M(x) = u_{\ep}^m(x)$ and $H_{e_j}(x) = u_{\ep,e_j}^h(x)$.

A formula for
the derivative $D_{e_j}A_{\varepsilon}(\omega)$ is found in \cite{Senn95}, and in our
case becomes
\begin{equation} \label{DAe}
\begin{split}
D_{e_j}A_{\varepsilon}(\omega) &= \int_{[0,1]^{d-1}}\int_{-\infty}^{\infty}
\frac{1}{2}|\nabla M|^2 + \ep V(x,M)
-\frac{1}{2}|\nabla H_{e_j}|^2 - \ep V(x,H_{e_j}) \,
dx_j \, dx^{d-1}\\
D_{-e_j}A_{\varepsilon}(\omega) &= \int_{[0,1]^{d-1}}\int_{-\infty}^{\infty}
\frac{1}{2}|\nabla H_{-e_j}|^2 + \ep V(x,H_{-e_j})
 -\frac{1}{2}|\nabla M|^2 -\ep V(x,M) \,  dx_j \, dx^{d-1}.
\end{split}
\end{equation}
Thus, the derivative of $A_{\varepsilon}(\omega)$ in the direction $e_j$ is
difference in the energies
of the minimizer defining the top of each connected component of
gap and the minimizing heteroclinic in the
direction $e_j$.
Note that $|\nabla (M+1)|^2 + \ep V(x,M+1) = |\nabla M|^2 + \ep V(x,M)$.

\subsection{Two dimensional case} \label{twoD}
We now focus on the two dimensional case so that we can compare
with the numerical computations, which were done for the
two dimensional case. The computations in higher dimensions
are impractical for us at the moment.

In our example, $\nabla H_{\eta} = k + \eta \sqrt{2\ep}/
\cosh(\sqrt{2\ep}\pi \eta x) -k\ep \cos(4\pi kx)/4|k|^2$.
If we select $\eta = e_1$ and compute $D_{e_1}S(k)$ we have
\[
\frac{1}{2}|\nabla H_{e_1}|^2 = \frac{1}{2}|k|^2 + \frac{k_1\sqrt{2\ep }}
{\cosh(\sqrt{2\ep \ }\pi x_1)} + \frac{ \ep}{\cosh^2(\sqrt{2\ep \ }\pi x_1)} -
\frac{\ep}{4} \cos(4\pi kx) + O(\ep^{3/2}).
\]
The heteroclinic solution $H_{-e_1}$ has the same expression as
above, with a sign change on the second term. The  $|k|^2$ terms  from
$|\nabla H_{\pm e_1}|$ and $|\nabla M| = |k|^2 + O(\ep)$ will cancel
when computing
$D_{\pm e_1}A_{\varepsilon}(\omega) $. When computing the jump in the derivative, that is
the sum $D_{e_1}A_{\varepsilon}(\omega) +D_{-e_1}A_{\varepsilon}(\omega) $,
the terms $\pm \frac{k_1\sqrt{2\ep}}
{\cosh(\sqrt{2\ep \ }\pi x_1)}$ will cancel, and we are left with two terms
of the form $\frac{ \ep}{\cosh^2(\sqrt{2\ep \ }\pi x_1)}$. These
two terms then contribute for a total of
\[
\int_{\mathbb{R}} \frac{2 \ep}{\cosh^2(\sqrt{2\ep \ }\pi x_1)}\, dx_1
\]
If we integrate over all of $\mathbb{R}$ then this integral is
\[
\int_{-\infty}^{\infty} \frac{2 \ep}{\cosh^2(\sqrt{2\ep \ }\pi x_1)}\, dx_1
= \frac{2\sqrt{2}}{\pi}\ep^{1/2}.
\]
Computing the contributions from the potentials,
$\int_{[0,1]^{d-1}} \ep V(x,M) - \ep V(x,H_{\pm e_1}) \,
dx_j \, dx^{d-1}$, is not easy analytically, but numerically we find
that they also yield $ \frac{2\sqrt{2}}{\pi}\ep^{1/2}$. Thus, the
jump in the derivative is $D_{e_1}A_{\varepsilon}(k)+D_{-e_1}A_{\varepsilon}(k) =
\frac{4\sqrt{2}}{\pi}\ep^{1/2} \approx 1.8\ep^{1/2}$, which agrees
well with the numerical computations.


\subsection{Comparison with results from the numerical
computations}
\begin{figure}[!htb]
\begin{center}
\includegraphics[height=1.8in]{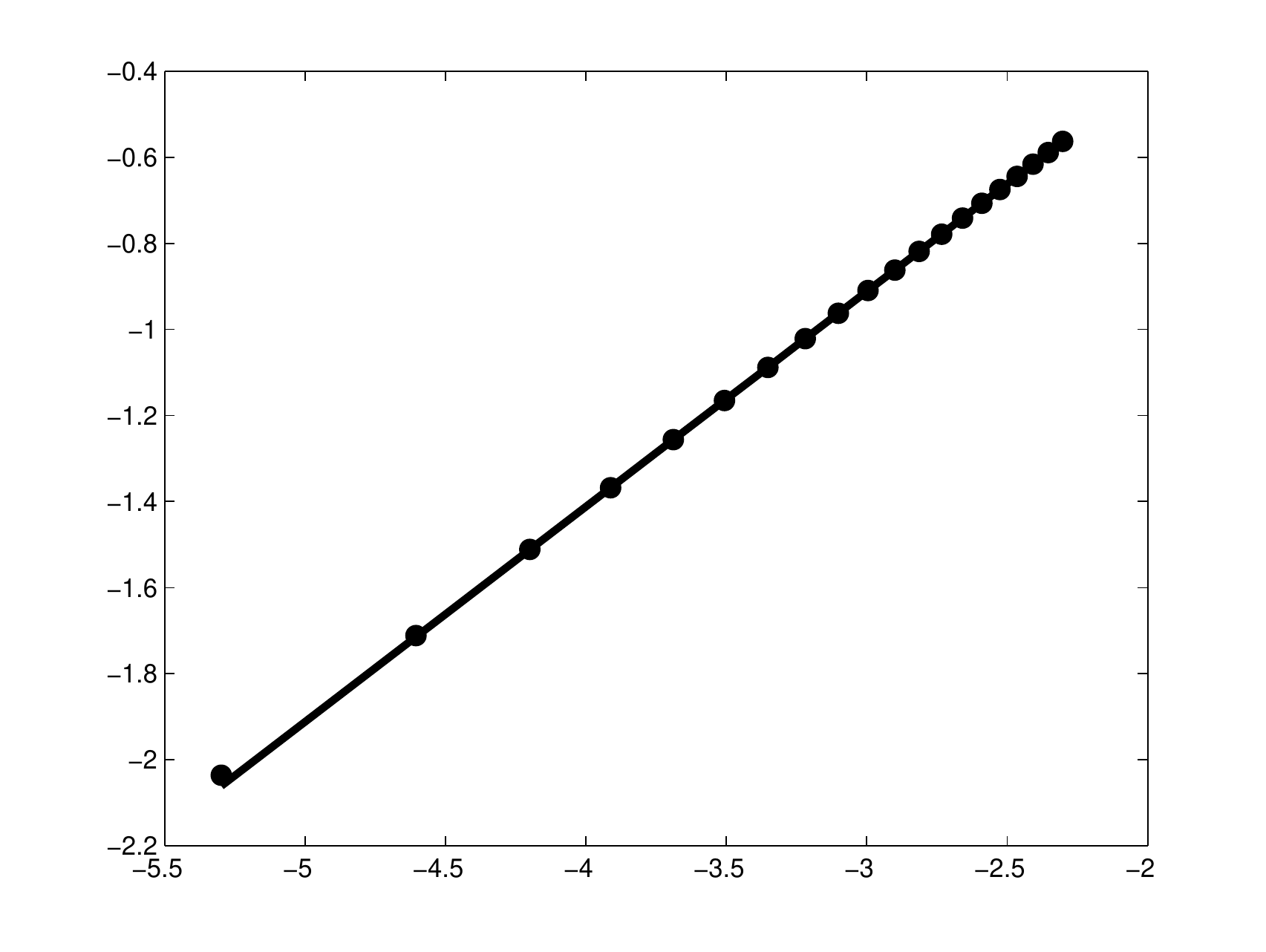}
\caption{Logarithm of the jump in
$D_{e_1}A_{\varepsilon}(\omega)$ against $\log(\ep)$ in the two dimensional example:
$V(x,u) = \ep \sin(2\pi kx)\cos(2\pi u)$,
where $k=(2,3)$. This is a ``first order resonance'' so
$\omega= k$.
256 modes were used for each
Fourier frequency direction $\xi_1$, $\xi_2$. }
\label{fig1}
\end{center}
\end{figure}
\begin{figure}[!htb]
\begin{center}
\subfigure{\includegraphics[scale=.34]{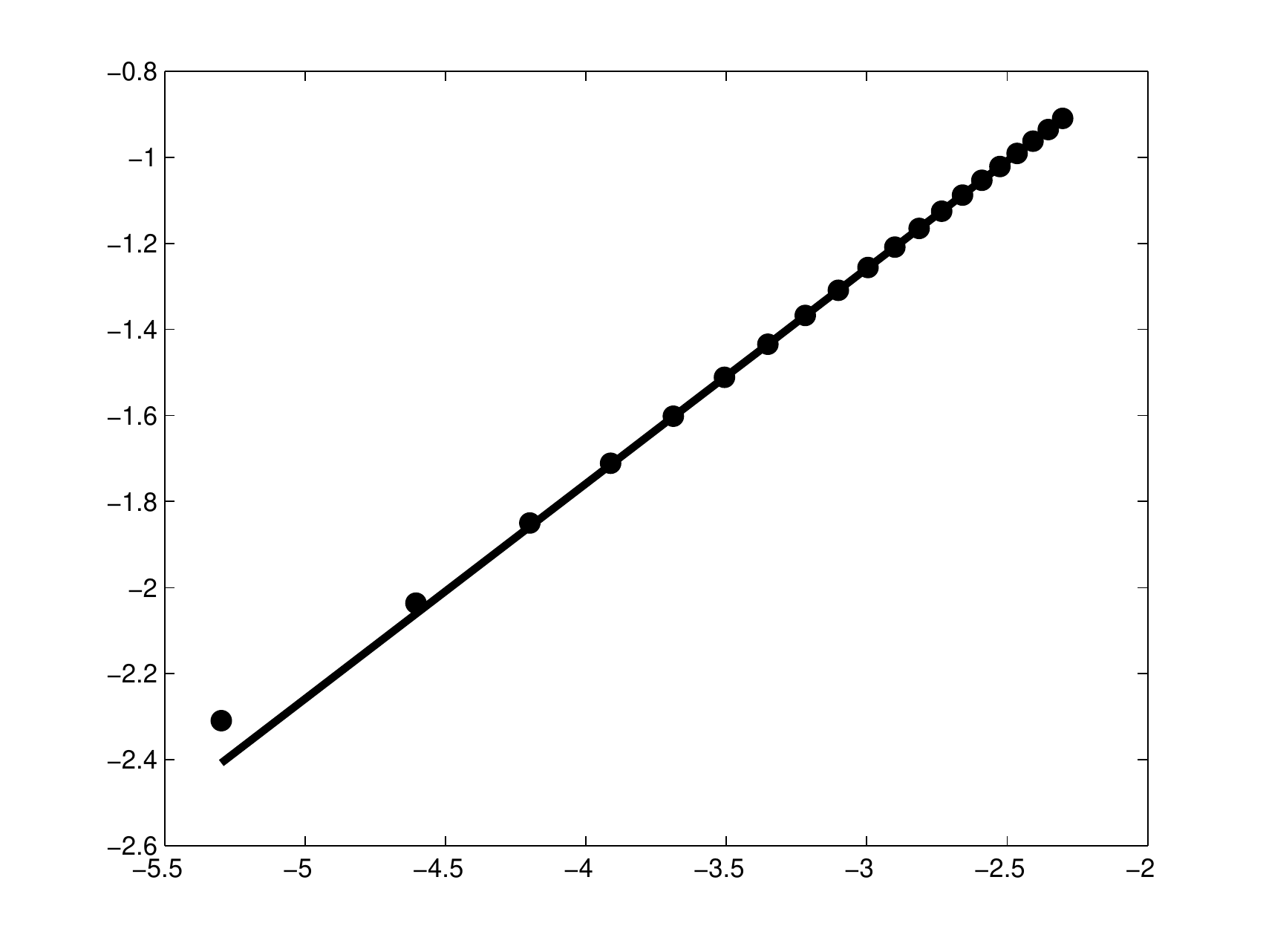}\label{pic1}}
\subfigure{\includegraphics[scale=.34]{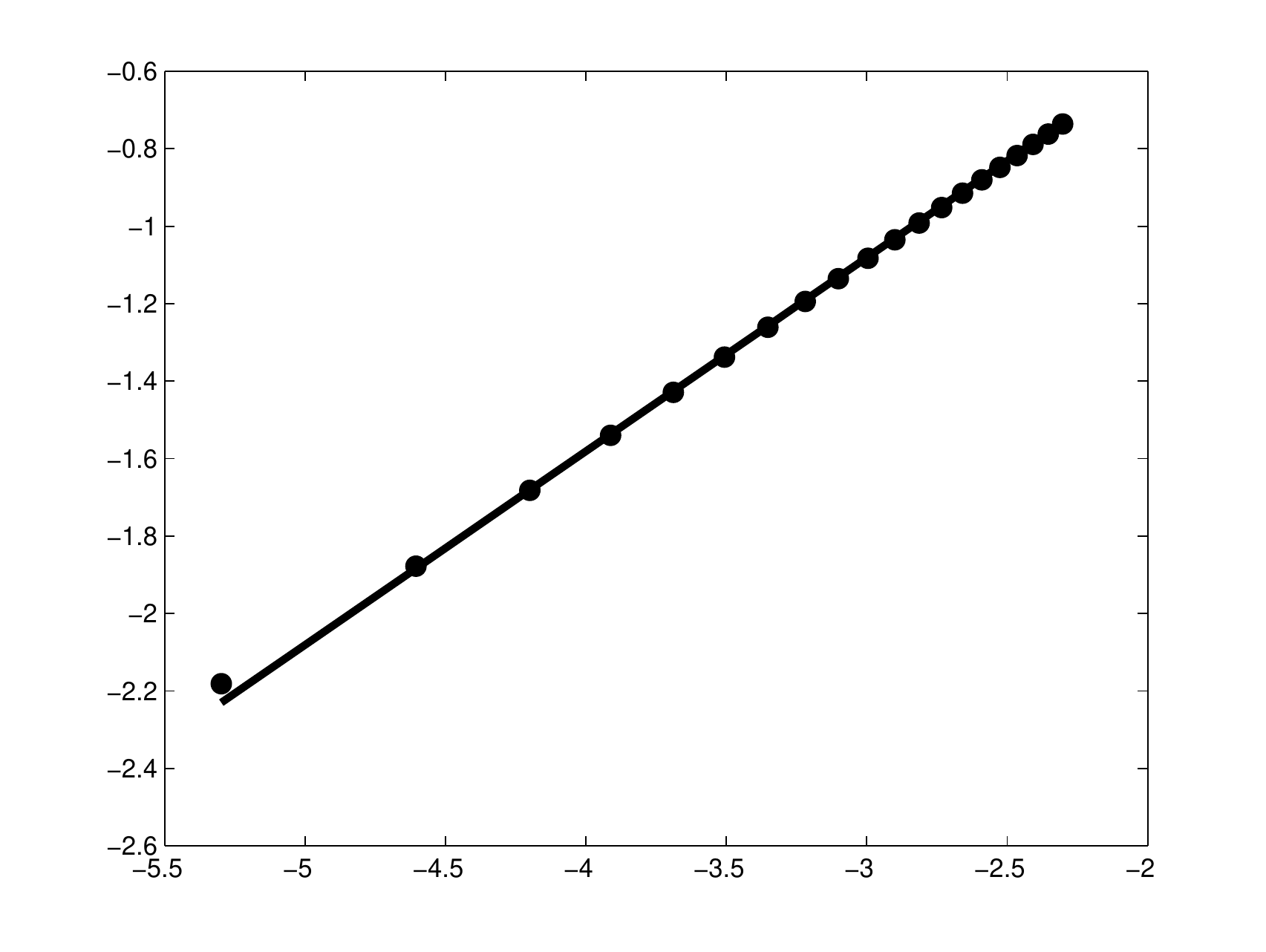}\label{pic2}}
\caption{Plots of the logarithm of the jump in
$D_{e_1}A_{\varepsilon}(\omega)$ against $\log(\ep)$
for different potential functions.
$V(x,u) = \ep \sin(2\pi k_1 x_1)\sin(2\pi k_2 x_2) \cos(2\pi u)$
on the left and
$V(x,u) = \frac{\ep}{2} \sin(2\pi kx)(\cos(2\pi u) + \sin (2\pi u))$
on the right. $k=(2,1)$ in each case, and $\omega= k$.
256 modes were used for each
Fourier frequency direction $\xi_1$, $\xi_2$.}
\label{fig2}
\end{center}
\end{figure}
We use the Sobolev gradient method to compute the minimizers
$u_\ep$ for several values of $\ep$, where each $u_\ep$ has
rotation vector $\omega$. We then repeat this computation for the same values
of $\ep$, but for rotation vectors $\omega\pm \Delta \omega e_j$.
With these minimizers
we can compute $A_{\varepsilon}(\omega)$ for each value of $\ep$ and each rotation vector.
Taking differences over the $\omega$-variable, we can approximate
the derivative of $A_{\varepsilon}(\omega)$ with respect to $\omega$.
\[
D_{e_j}A_{\varepsilon}(\omega) \approx \frac{A_\ep(\omega+\Delta \omega e_j)
-A_\ep(\omega)}{\Delta \omega}.
\]
The plots in Figures \ref{fig1} and \ref{fig2}
 are for three cases of potential function.
In each case we examine a first-order resonance, so the choice of $\alpha$
in the Lindstedt series is made when solving the order $\ep$ equation.
This also means the twist condition (\ref{twist}) is satisfied.

Figure \ref{fig1} is log-log plot of the jump in the
$e_1$ direction of the gradient of $A_{\varepsilon}(\omega)$
versus $\ep$.
The dotted line is the logarithms of the numerically computed points
plotted against $\log(\ep)$. The solid line is the
fit $J(\ep) = \log(\frac{4\sqrt{2}}{\pi}\ep^{1/2})$, which was derived
in Section \ref{twoD}. The same relation is found for the
resonance at $\omega= -k$.

Figure \ref{fig2} has the same type of plot
for different
choices of potential function.
In \ref{pic1} the potential is
\[
V(x,u) = \ep \sin(2\pi k_1 x_1)\sin(2\pi k_2 x_2) \cos(2\pi u),
\]
and the choice of resonance was $\omega= k = (2,1)$.
In this case, the fit is $J(\ep) = \log(\frac{4}{\pi}\ep^{1/2})$.
The same
relation is found at $\omega= (\pm k_1,\pm k_2)$ and
at $\omega= (\pm k_2, \pm k_1)$.

In \ref{pic2} the potential is
\[
V(x,u) = \frac{\ep}{2} \sin(2\pi kx)(\cos(2\pi u) + \sin (2\pi u)),
\]
and the choice of $\omega$ was $\omega= k = (2,1)$.
The same relation
is found at $\omega= -k$.

Higher order resonances can be computed, for instance at $\omega=2k$ in
each of the previous examples. However, the behavior of the jump in
$DA_{\varepsilon}(\omega)$ behaves like $\ep$ to a power greater than one. To get reasonable
accuracy one needs to take many more Fourier modes in the numerical
approximation, which is impractical for us at the moment
in the case of two spatial dimensions.



\def\cprime{$'$}

\medskip
Received January 2011; revised April 2011.
\medskip


\begin{thebibliography}{99}
\expandafter\ifx\csname urlstyle\endcsname\relax
  \providecommand{\doi}[1]{doi:\discretionary{}{}{}#1}\else
  \providecommand{\doi}{doi:\discretionary{}{}{}\begingroup
  \urlstyle{rm}\Url}\fi

\bibitem[Ban87a]{Bangert87b}(MR921095)
V.~Bangert,
\newblock \emph{The existence of gaps in minimal foliations,}
\newblock Aequationes Math., \textbf{34} (1987), 153--166.


\bibitem[Ban87b]{Bangert87a}(MR920054)
V.~Bangert,
\newblock \emph{A uniqueness theorem for {${\bf Z}^n$}-periodic variational problems,}
\newblock Comment. Math. Helv., \textbf{62} (1987), 511--531.


\bibitem[Ban89]{Bangert89}(MR991874)
V.~Bangert,
\newblock \emph{On minimal laminations of the torus,}
\newblock Ann. Inst. H. Poincar\'e Anal. Non Lin\'eaire, \textbf{6} (1989), 95--138.

\bibitem[BLV11]{BlassLlaveVal10}
T.~Blass, R.~de~la Llave and E.~Valdinoci,
\newblock \emph{A comparison principle for a {S}obolev gradient semi-flow,}
\newblock Commun. Pure Appl. Anal., \textbf{10} (2011), 69--91.
\newblock \doi{10.3934/cpaa.2011.10.69}.

\bibitem[CF96]{Chierchia}(MR1385915)
L.~Chierchia and C.~Falcolini,
\newblock \emph{A note on quasi-periodic solutions of some elliptic systems,}
\newblock Z. Angew. Math. Phys., \textbf{47} (1996), 210--220.


\bibitem[Eva98]{Evans98}(MR1625845)
L.~C. Evans,
\newblock ``Partial Differential Equations,'' volume~\textbf{19} of ``Graduate
  Studies in Mathematics,"
\newblock American Mathematical Society, Providence, RI, 1998.


\bibitem[Kat66]{Kato66}(MR0203473)
T.~Kato,
\newblock ``Perturbation Theory for Linear Operators,''
\newblock Die Grundlehren der mathematischen Wissenschaften, Band 132.
  Springer-Verlag New York, Inc., New York, 1966.


\bibitem[LV09]{LlaveVal09}(MR2542727)
R.~de~la Llave and E.~Valdinoci,
\newblock \emph{A generalization of {A}ubry-{M}ather theory to partial differential
  equations and pseudo-differential equations,}
\newblock Ann. Inst. H. Poincar\'e Anal. Non Lin\'eaire,
 \textbf{26} (2009), 1309--1344.


\bibitem[Mor73]{Morse}(MR0420368)
M.~Morse,
\newblock ``Variational Analysis: Critical Extremals and {S}turmian
  Extensions,"
\newblock Interscience Publishers [John Wiley \& Sons, Inc.], New
  York-London-Sydney, 1973.


\bibitem[Mos86]{Moser86}(MR847308)
J.~Moser,
\newblock \emph{Minimal solutions of variational problems on a torus,}
\newblock Ann. Inst. H. Poincar\'e Anal. Non Lin\'eaire, \textbf{3} (1986), 229--272.


\bibitem[Neu10]{Neuberger10}(MR2573187)
J.~W. Neuberger,
\newblock ``Sobolev Gradients and Differential Equations," volume \textbf{1670} of
  ``Lecture Notes in Mathematics,"
\newblock Springer-Verlag, Berlin, second edition, 2010.


\bibitem[Nir59]{Nirenberg59}(MR0109940)
L.~Nirenberg,
\newblock \emph{On elliptic partial differential equations,}
\newblock Ann. Scuola Norm. Sup. Pisa(3), {\bf 13} (1959), 115--162.


\bibitem[Sen91]{Senn91}(MR1094738)
W.~Senn,
\newblock \emph{Strikte {K}onvexit\"at f\"ur {V}ariationsprobleme auf dem
  {$n$}-dimensionalen {T}orus,}
\newblock Manuscripta Math., \textbf{71} (1991), 45--65.


\bibitem[Sen95]{Senn95}(MR1385292)
W.~M. Senn,
\newblock \emph{Differentiability properties of the minimal average action,}
\newblock Calc. Var. Partial Differential Equations, \textbf{3} (1995), 343--384.


\end{thebibliography}
\end{document}